\documentclass{amsart}
\usepackage{amsthm,amscd,amssymb,verbatim,epsf,amsmath,eurosym,amsfonts,mathrsfs,graphicx}
\usepackage[colorlinks=true,linkcolor=blue,citecolor=blue]{hyperref}
\usepackage{amsmath, amsthm, amscd, amsfonts, amssymb, graphicx, color}
\makeatletter
\@namedef{subjclassname@2020}{\textup{2020} Mathematics Subject Classification}
\makeatother

\begin{document}

\makeatletter
\newcommand{\verbatimfont}[1]{\renewcommand{\verbatim@font}{\ttfamily#1}}
\makeatother
\newcommand{\R}{\mathbb{R}}
\newcommand{\C}{\mathbb{C}}
\newcommand{\om}{\omega}
\newcommand{\I}{{\mbox{I}}}
\newcommand{\II}{{\mbox{II}}}
\newcommand{\III}{{\mbox{III}}}
\newcommand{\Hess}{\mbox{\rm{Hess}}}
\newcommand{\Sing}{\mbox{\rm{Sing}}}
\newcommand{\Ind}{\mbox{\rm{Ind}}}
\theoremstyle{plain}
\newtheorem{Thm}{Theorem}
\newtheorem{Cor}{Corollary}
\newtheorem{Ex}{Example}
\newtheorem{Con}{Conjecture}
\newtheorem{Main}{Main Theorem}
\newtheorem{Lem}{Lemma}
\newtheorem{Prop}{Proposition}

\theoremstyle{definition}
\newtheorem{Def}{Definition}
\newtheorem{Note}{Note}
\newtheorem{Question}{Question}

\newtheorem{remark}{Remark}
\newtheorem{notation}{Notation}
\renewcommand{\thenotation}{}

\renewcommand{\rm}{\normalshape}%

\title{A Note on Umbilic Points at Infinity}
\begin{abstract}
In this note a definition of {\em umbilic point at infinity} is proposed, at least for surfaces that are homogeneous polynomial graphs over a plane in Euclidean 3-space. This is a stronger definition than that of Toponogov in his study of complete convex surfaces, and allows one to distinguish between different umbilic points at infinity. 

It is proven that all such umbilic points at infinity are isolated, that they occur in pairs and are the zeroes of the projective extension of the third fundamental form,  as developed in \cite{gao}. A geometric interpretation for our definition is that an umbilic point at infinity occurs when the tangent to the level set at infinity is also an asymptotic direction at infinity.

\end{abstract}
\author{Brendan Guilfoyle}
\address{Brendan Guilfoyle\\
          School of STEM \\
          Munster Technological University, Kerry\\
          Tralee  \\
          Co. Kerry \\
          Ireland.}
\email{brendan.guilfoyle@mtu.ie}

\date{\today}
\maketitle


\section{Isolated Umbilic Points}

Points on a surface in ${\mathbb R}^3$ where the principal curvatures are equal, called {\it umbilic points}, play a central role in many areas of classical surface theory. For example, numerous results exist about totally umbilic surfaces and their generalizations  \cite{Hopf_book} \cite{sanini} \cite{soum_toubiana} \cite{spivak}. Isolated umbilic points in particular have mysterious properties. The distinction between smoothness and real analyticity for a surface is intimately related to its umbilic points, as has been observed from both dynamic and stationary perspectives \cite{guil} \cite{gak_hopf} \cite{HandW1} \cite{Hopf}. Their indices have been found to have applications in the study of zeros of holomorphic polynomials \cite{gak_roots}.

The purpose of this paper is to extend the notion of umbilic point to {\em umbilic point at infinity} in the case where the surface in question is a graph over a plane.The principal foliations at infinity have been considered previously by Ando using a one-point compactification of the surface \cite{ando}. In \cite{gao}, the authors defined umbilic points at infinity for more general polynomial graphs via the zeros of the projective extension of the third fundamental form.

This paper considers an alternative definition in terms of the fall-off of the curvature at infinity. We show that this agrees with the definition using the projective extension described in \cite{gao}. Moreover we show that such an umbilic point at infinity can be characterised as a point at infinity where an asymptotic direction becomes tangent to the level set of the defining function.

In the case where the surface is given by a homogeneous polynomial graph, the points at infinity can be identified with a circle and asymptotic geometric invariants can be computed relatively easily. Moreover, umbilic points at infinity obtain a real projective significance.

In \cite{Top} Toponogov says that a complete convex plane $P\subset{\mathbb R}^3$ must have {\em an umbilic at infinity} if it has no finite umbilic points but
\begin{equation}\label{d:top_umatinf}
\inf_{p\in P}|\kappa_1(p)-\kappa_2(p)|=0,
\end{equation}
where $\kappa_1,\kappa_2$ are the principal curvatures of $P$.

In the same paper Toponogov conjectures that every complete convex plane must have an umbilic point, albeit at infinity in his sense, i.e. equality (\ref{d:top_umatinf}) must hold.  Graphs of polynomial functions are easily seen to satisfy (\ref{d:top_umatinf}), so Toponogov's conjecture holds for such surfaces. A proof of Toponogov's conjecture is proposed in \cite{gak_top} which uses Fredholm regularity of an associated boundary value problem and mean curvature flow with boundary in co-dimension 2.

In this paper we adopt a stronger definition than Toponogov's for homogeneous polynomial graphs. To set notation and assist in the discussion, let us adopt the following definitions here and throughout.   

$P$ is a complete embedding of ${\mathbb R}^2$ in ${\mathbb R}^3$ defined by a graph $x^3=f(x^1,x^2)$ for some polynomial function $f$.  The Gauss map $\pi:P\rightarrow {\mathbb S}^2$ takes a point on the surface in ${\mathbb R}^3$ to the direction of its unit normal. 

Moving to complex coordinates $z=x^1+ix^2$ on ${\mathbb R}^2$ so that $x^3=f(z,\bar{z})$ and holomorphic coordinate $\xi$ about the North pole on ${\mathbb S}^2$ in the image of the Gauss map, the following relation is easily seen to hold:
\begin{equation}\label{e:gmaprel}
\frac{\xi}{1-\xi\bar{\xi}}=-\partial_{\bar{z}}f.
\end{equation}
As a point goes out to infinity on the surface, the normal generally tips over to the equator $|\xi|=1$. Denote this circle by ${\mathbb S}^1_\infty\subset{\mathbb S}^2$. This defines a distance to infinity which we use in the sequel.
\vspace{0.1in}
\begin{Def}
The {\em distance $d_\infty(p)$ of a point $p$ to infinity} is the spherical distance of $\pi(p)$ to ${\mathbb S}^1_\infty$ in ${\mathbb S}^2$. In holomorphic coordinates, the distance to infinity is simply $d_\infty=1-|\xi|$.
\end{Def}
\vspace{0.1in}
We can now define what it means to have an umbilic point at infinity.
\vspace{0.1in}
\begin{Def}\label{d:umdata}
$P$ has an {\em umbilic at infinity at $A\in [0,2\pi)$} if
\[
\lim_{R\rightarrow \infty}\left.\frac{|\kappa_1-\kappa_2|}{d_\infty}\right|_{z=Re^{iA}}=0.
\]
\end{Def}
\vspace{0.1in}
As $p$ goes to infinity in $P$, $d_\infty(p)$ goes to 0, and so this is a stronger requirement than Toponogov's.  To have an umbilic point at infinity by our definition is to have a direction along which the difference between the curvatures goes to zero at least as fast as the Gauss map approaches the equator.

As we will see in Section \ref{s6}, the elliptic paraboloid would have an umbilic at infinity according to Toponogov, but has no umbilic points at infinity according to our definition. 

A geometric interpretation for our definition given below is that an umbilic point at infinity occurs when the tangent to the level set at infinity is also an asymptotic direction at infinity.

For homogeneous polynomial graphs we prove: 
\vspace{0.1in}
\begin{Thm}\label{t:1}
The graph of any homogeneous polynomial has a finite, even number of umbilic points at infinity. If the polynomial has degree $n$, then the number of umbilic points at infinity is no more than $6n-8$.
\end{Thm}
\vspace{0.1in}
Note that the upper estimate on the number of umbilic points at infinity can be improved  and refined - see \cite{gao}. We prove the stronger analog of Toponogov's conjecture for homogeneous graphs:
\vspace{0.1in}
\begin{Thm}\label{t:2}
A homogeneous polynomial graph $P$ over a plane in ${\mathbb R}^3$ satisfies 
\[
\inf_{p\in P}\frac{|\kappa_1(p)-\kappa_2(p)|}{d_\infty(p)}=0.
\]

\end{Thm}
\vspace{0.1in}

This paper is organized as follows. In the next section the necessary background on surface geometry is covered and Section \ref{s3} applies this to smooth graphs. In Section \ref{s4} umbilic points at infinity are characterized for homogeneous polynomial graphs. 

Section \ref{s5} contains the proofs of Theorems \ref{t:1} and \ref{t:2}. In the final section the various geometric quantities are computed for elliptic and hyperbolic paraboloids and their umbilic points are infinity computed.

\vspace{0.1in}

\section{Background}\label{s2}

Our approach to describing surfaces in ${\mathbb R}^3$ and their curvatures is through the geometry of the space of oriented lines of ${\mathbb R}^3$ and its submanifolds \cite{gak2} \cite{gak5}. In this we seek properties that do {\em not} generalize to higher dimensions, phenomena that only occur for surfaces in three dimensions. 

Complex coordinates are useful as they allow one to more easily track winding numbers and count indices, and they have a special role in two dimensions. Most of what follows in this section can be found in the above cited papers and is summarized to fix notation and establish canonical relations.

Suppose now that $S$ is a smooth surface in ${\mathbb R}^3$ with second fundamental form $\Pi :TS\times TS\rightarrow {\mathbb R}$. On the convex parts of $S$ there may be umbilic points, namely, points where the eigenvalues of $\Pi$ are equal: $\kappa_1=\kappa_2$. These are generically isolated. At every non-umbilic point the eigen-directions of $\Pi$ form a pair of orthogonal lines, called the principal directions.      

The principal directions determine the principal foliations on $S$, which have singularities at the umbilic points. For an isolated umbilic point $p\in S$ there is a well-defined half-integer index determined by the principal foliation around $p$: ${\mbox{Ind }}(p)\in{\mathbb Z}/2$.

In order to compute these, start with the space of oriented lines in ${\mathbb R}^3$, which is modeled by $T{\mathbb S}^2$. Choosing holomorphic coordinates $\xi$ about the North pole of ${\mathbb S}^2$, construct complex coordinates $(\xi,\eta)$ on $T{\mathbb S}^2$ by identifying these complex numbers with the tangent vector $X\in T_\xi{\mathbb S}^2$ given by
\[
X=\eta\frac{\partial}{\partial\xi}+\bar{\eta}\frac{\partial}{\partial\bar{\xi}} .
\]

\begin{center}
\end{center}

Now given a parameterized surface $S$ in ${\mathbb R}^3$ the oriented normal lines to $S$ form a parameterized surface in $T{\mathbb S}^2$. Let the complex variable $\mu$ be the parameter on both $S$ and its normal lines.
That is, $S$ has oriented normal lines $(\xi,\eta)=(\xi(\mu,\bar{\mu}), \eta(\mu,\bar{\mu}))$,
where $\xi$ and $\eta$ are functions of $\mu$ and $\bar{\mu}$.

To compute the second fundamental form explicitly in terms of the derivatives of $\xi$ and $\eta$, choose an orthonormal frame $\{e_0,e_1,e_2\}$ along $S$ in ${\mathbb R}^3$ adapted to the surface $S$ in that $e_0$ is normal to $S$ while $e_1$ and $e_2$ are tangent. Clearly this is only well-defined up to a rotation about the normal, which, if we introduce complex frames
\[
e_+={\textstyle{\frac{1}{\sqrt{2}}}}\left(e_1+ie_2\right) \qquad\qquad e_-={\textstyle{\frac{1}{\sqrt{2}}}}\left(e_1-ie_2\right),
\]
is given by $(e_+,e_-)\rightarrow (e^{i\theta}e_+,e^{-i\theta}e_-)$. 
Define the complex quantities
\begin{equation}\label{e:spinco_def}
\sigma={\mbox{ II}}(e_+,e_+) \qquad\qquad \rho={\mbox{ II}}(e_+,e_-) ,
\end{equation}
where $\sigma$ is complex valued and $\rho$ is real.

The significance of these for surface theory is the following:

\vspace{0.1in}
\begin{Prop}\cite{gak5}
\[
|\sigma|={\textstyle{\frac{1}{2}}}|\kappa_1-\kappa_2| \qquad\qquad \rho={\textstyle{\frac{1}{2}}}(\kappa_1+\kappa_2),
\]
and the argument of $\sigma$ determines the principal directions, as measured in the complex coordinates. 
\end{Prop}
\vspace{0.1in}
Umbilic points occur when $\sigma=0$ and, if $\sigma=H\xi^n\bar{\xi}^m$ for $n,m\in{\mathbb N}$ and $H$ a real non-zero function, then the origin would be an isolated umbilic point of index $(m-n)/2$. 

Turn now to the local computation of the second fundamental form in terms of these quantities:

\vspace{0.1in}
\begin{Prop} \label{p:bgnd}\cite{gak2}
The curvatures have the following expressions in terms of the first derivatives of the parameterization in oriented line space:  
\begin{equation}\label{e:spinco}
\rho=\frac{ \partial^+\eta\overline{\partial}\;\overline{\xi} -\partial^-\eta\partial\overline{\xi}}
{\partial^-\eta\overline{\partial^-\eta}-\partial^+\eta\overline{\partial^+\eta}}
\qquad\qquad
\sigma=\frac{\overline{\partial^+\eta}\partial\overline{\xi} -\overline{\partial^-\eta}\;\overline{\partial}\;\overline{\xi}}
{\partial^-\eta\overline{\partial^-\eta}-\partial^+\eta\overline{\partial^+\eta}},
\end{equation}
where
\begin{equation}\label{e:spincoa}
\partial^+\eta\equiv\partial \eta+\left(r-\frac{2\eta\overline{\xi}}{1+\xi\overline{\xi}}\right)\partial \xi
\qquad\qquad
\partial^-\eta\equiv\overline{\partial} \eta+\left(r-\frac{2\eta\overline{\xi}}{1+\xi\overline{\xi}}\right)\bar{\partial} \xi,
\end{equation}
and $\partial$ and $\bar{\partial}$ are differentiation with respect to $\mu$ and $\bar{\mu}$, respectively.

A 2-parameter family of oriented lines $(\xi(\mu,\bar{\mu}),\eta(\mu,\bar{\mu}))$ is orthogonal to a surface in ${\mathbb{R}}^3$ iff there exists a real function $r(\mu,\bar{\mu})$ satisfying:
\begin{equation}\label{e:intsur}
\bar{\partial} r=\frac{2\eta\bar{\partial}\bar{\xi}+2\bar{\eta}\bar{\partial}\xi}{(1+\xi\bar{\xi})^2}.
\end{equation}
This function is classically called the {\em support function} of $S$. If there exists one solution $r$ to this equation, then there exists a 1-parameter family of parallel surfaces, generated by the real constant of integration.
\end{Prop}
\vspace{0.1in}

\section{Smooth Graphs}\label{s3}

Consider then a plane $P$ embedded in ${\mathbb R}^3$ that is a graph $x^3=f(x^1,x^2)$ over ${\mathbb R}^2$. 
Introduce complex coordinates $z=x^1+ix^2$. The basic geometry of such surfaces is captured by:

\begin{Prop}\label{p:graph}
The support function of a graph is 
\begin{equation}\label{e:support_pl}
r=\frac{-f+z\partial_zf+\bar{z}\partial_{\bar{z}}f}{[1+4\partial_zf\partial_{\bar{z}}f]^{\scriptstyle{\frac{1}{2}}}},
\end{equation}
the normal direction is
\begin{equation}\label{e:xi_pl}
\xi=\frac{-2\partial_{\bar{z}}f}{1-[1+4\partial_zf\partial_{\bar{z}}f]^{\scriptstyle{\frac{1}{2}}}},
\end{equation}
and  the oriented normal line $(\xi,\eta)$ is
\begin{equation}\label{e:eta_pl}
\eta=\frac{2\partial_{\bar{z}}f(z\partial_zf-\bar{z}\partial_{\bar{z}}f)}{(1-[1+4\partial_zf\partial_{\bar{z}}f]^{\scriptstyle{\frac{1}{2}}})^2}
+\frac{z+2f\partial_{\bar{z}}f}{1-[1+4\partial_zf\partial_{\bar{z}}f]^{\scriptstyle{\frac{1}{2}}}}.
\end{equation}
\end{Prop}
\begin{proof}
We  use coordinates $(z=x^1+ix^2,\bar{z}=x^1-ix^2,x^3)$ on ${\mathbb R}^3={\mathbb C}\times {\mathbb R}$ in which the Euclidean metric is
\[
g=\left[\begin{matrix}
0 &{\textstyle{\frac{1}{2}}} & 0 \\
{\textstyle{\frac{1}{2}}} & 0 & 0 \\
0 &0 & 1 
\end{matrix}\right].
\]
For any surface given by $x^3=f(x^1,x^2)=f(z,\bar{z})$, the unit normal is easily found to be
\begin{align}
\hat{N}&=-\frac{2\partial_{\bar{z}}f}{[1+4|\partial_zf|^2]^{\scriptstyle{\frac{1}{2}}}}\frac{\partial}{\partial z}-\frac{2\partial_zf}{[1+4|\partial_zf|^2]^{\scriptstyle{\frac{1}{2}}}}\frac{\partial}{\partial \bar{z}}+\frac{1}{[1+4|\partial_zf|^2]^{\scriptstyle{\frac{1}{2}}}}\frac{\partial}{\partial x^3}\nonumber\\
&=\frac{2\xi}{1+\xi\bar{\xi}}\frac{\partial}{\partial z}+\frac{2\bar{\xi}}{1+\xi\bar{\xi}}\frac{\partial}{\partial \bar{z}}+\frac{1-\xi\bar{\xi}}{1+\xi\bar{\xi}}\frac{\partial}{\partial x^3}\nonumber,
\end{align}
where the last equality identifies any unit vector in ${\mathbb R}^3$ with its complex coordinate $\xi$ on the unit ${\mathbb S}^2$. 

This proves equation (\ref{e:gmaprel}), and equation (\ref{e:xi_pl}) is just a rearrangement of this relation. 

Note how the image of the Gauss map approaches the equator $|\xi|=1$ at exactly the points where the slope of the function blows up. 

Equation (\ref{e:eta_pl}) follows from substituting the expression (\ref{e:xi_pl}) for $\xi$ into the incidence relationship for points and lines:
\[
\eta={\textstyle{\frac{1}{2}}}\left(x^1+ix^2-2x^3\xi-(x^1-ix^2)\xi^2\right)
     ={\textstyle{\frac{1}{2}}}\left(z-2f\xi-\bar{z}\xi^2\right).
\]

Finally, the support function in equation (\ref{e:support_pl}) is easily found to solve equation (\ref{e:intsur}).

\end{proof}
\vspace{0.1in}
\begin{Prop}
The complex quantities defined in (\ref{e:spinco_def}) which determine the principal curvatures and directions take the following form for a graph:
\vspace{0.1in}
\begin{equation}\label{e:shear_pl}
\sigma=\frac{2(1-\xi\bar{\xi})}{(1+\xi\bar{\xi})^3}\left(\bar{\xi}^{-2}\partial_z\partial_zf-2\partial_z\partial_{\bar{z}}f+\bar{\xi}^2\partial_{\bar{z}}\partial_{\bar{z}}f\right)\bar{\xi}^2
\end{equation}
\begin{equation}\label{e:div_pl}
\rho=\frac{2(1-\xi\bar{\xi})}{(1+\xi\bar{\xi})^3}\left(\xi^2\partial_z\partial_zf-(1+\xi^2\bar{\xi}^2)\partial_z\partial_{\bar{z}}f+\bar{\xi}^2\partial_{\bar{z}}\partial_{\bar{z}}f\right).
\end{equation}
As a consequence, the Gauss curvature is
\begin{equation}\label{e:gauss_curv}
\kappa=4\left(\frac{1-\xi\bar{\xi}}{1+\xi\bar{\xi}}\right)^4\left[\left(\partial_z\partial_{\bar{z}}f\right)^2-\partial_z\partial_zf\partial_{\bar{z}}\partial_{\bar{z}}f\right].
\end{equation}
\end{Prop}
\begin{proof}
Equations (\ref{e:shear_pl}) and (\ref{e:div_pl}) follow by differentiation of the expressions given in (\ref{e:xi_pl})  and (\ref{e:eta_pl}), using the formulae in Proposition \ref{p:bgnd} for $\sigma$ and $\rho$. Here the general parameterization $\mu$ is replaced by the graph parameterization $\mu=z$. 

Equation (\ref{e:gauss_curv}) follows directly from equations (\ref{e:shear_pl}) and (\ref{e:div_pl}) via the fact that $\kappa=\rho\bar{\rho}-\sigma\bar{\sigma}$.
\end{proof}

\vspace{0.1in}
\section{Umbilic points at infinity}\label{s4}

The Gauss map at infinity takes any $A\in [0,2\pi)$ to
\[
\theta_\infty(A)={\textstyle{\frac{1}{2i}}}\lim_{R\rightarrow\infty}\ln\left.\left(\frac{\partial_{\bar{z}}f}{\partial_{{z}}f}\right)\right|_{z=Re^{iA}}={\textstyle{\frac{1}{2i}}}\lim_{R\rightarrow\infty}\ln\left.\left(\frac{\xi}{\bar{\xi}}\right)\right|_{z=Re^{iA}},
\]
if the limit exists. Note that if it is defined on the whole circle $\theta_\infty:S^1\rightarrow {\mathbb S}^1_\infty$, the Gauss map at infinity has a well defined winding number. This is not uniquely determined by the degree of the polynomial as, for example, the elliptic paraboloid has winding number $1$, while the hyperbolic paraboloid has winding number $-1$. 

Note also that
\begin{align}
\lim_{R\rightarrow\infty}|\kappa_1-\kappa_2|&=2\lim_{R\rightarrow\infty}|\sigma|={\textstyle{\frac{1}{2}}}\lim_{R\rightarrow\infty}(1-\xi\bar{\xi})|\bar{\xi}^{-2}\partial_z\partial_zf-2\partial_z\partial_{\bar{z}}f+\bar{\xi}^2\partial_{\bar{z}}\partial_{\bar{z}}f|\nonumber\\
&\qquad \qquad={\textstyle{\frac{1}{2}}}\lim_{R\rightarrow\infty}\frac{|\bar{\xi}^{-2}\partial_z\partial_zf-2\partial_z\partial_{\bar{z}}f+\bar{\xi}^2\partial_{\bar{z}}\partial_{\bar{z}}f|}{|\partial_zf|}.\nonumber
\end{align}
Thus, for a polynomial graph a generic direction (such that $\lim_{R\rightarrow\infty}\partial_zf\neq0$) this will be well-defined and, in fact, zero. This shows that Toponogov's conjecture, in the sense that
\[
\inf_{p\in P}|\kappa_1(p)-\kappa_2(p)|=0,
\]
holds for any polynomial graph. However, it also says that every polynomial graph has an umbilic point at infinity, in the sense of Toponogov.

Let us now turn to our stronger Definition \ref{d:umdata} of umbilic points at infinity. First, consider the following complex quantity for $A\in [0,2\pi)$:
\[
\tilde{\sigma}_\infty(A)=\lim_{R\rightarrow \infty}\left.\frac{\sigma}{1-|\xi|^2}\right|_{z=Re^{iA}},
\]
where $\sigma$ is given by equation (\ref{e:shear_pl}) and the limit may not exist. 

The distance to infinity is $d_\infty=1-|\xi|$, and so $1-|\xi|^2=d_\infty(2-d_\infty)$. Since the second factor is bounded and non-zero, 
\[
\tilde{\sigma}_\infty(A)=0 \qquad\qquad {\mbox{ iff}} \qquad\qquad \lim_{R\rightarrow \infty}\left.\frac{|\kappa_1-\kappa_2|}{d_\infty}\right|_{z=Re^{iA}}=0.
\]
Thus if this limit, which may not exist, not only exists but is zero for some $A$, then we say, as per Definition \ref{d:umdata}, that the surface has an umbilic point at infinity at $A$.  

\vspace{0.1in}
\begin{Prop}\label{p:4}
A graphical surface $P$ has an umbilic point at infinity at $A$ if any one of the following equivalent conditions hold
\begin{itemize}
\item[(i)] 
\[
\lim_{R\rightarrow \infty}\left.\frac{\kappa_1-\kappa_2}{d_\infty}\right|_{z=Re^{iA}}=0,
\]
\item[(ii)] 
\[
\tilde{\sigma}_\infty(A)=0, 
\]
\item[(iii)] 
\[
\lim_{R\rightarrow \infty}\left.e^{2i\theta_\infty}\partial_z\partial_zf-2\partial_z\partial_{\bar{z}}f+e^{-2i\theta_\infty}\partial_{\bar{z}}\partial_{\bar{z}}f\right|_{z=Re^{iA}}=0,
\]
\item[(iv)] there is an asymptotic direction at infinity that is tangent to the level set at infinity.
\end{itemize}
\end{Prop}
\begin{proof}
We have already seen that $(i)\iff (ii)$. 

From equation (\ref{e:gmaprel}) one can see that at infinity $\xi\rightarrow e^{i\theta_\infty}$ as $R\rightarrow\infty$. Here $\theta_\infty$ can be viewed as a function of $A$.

Then from equation (\ref{e:shear_pl}) we see that
\[
4|\tilde{\sigma}_\infty|=\lim_{R\rightarrow \infty}\left.\frac{4|\sigma|}{1-|\xi|^2}\right|_{z=Re^{iA}}=\lim_{R\rightarrow \infty}\left|e^{2i\theta}\partial_z\partial_zf-2\partial_z\partial_{\bar{z}}f+e^{-2i\theta}\partial_{\bar{z}}\partial_{\bar{z}}f\right|_{z=Re^{iA}}
\]
This proves the that $(ii)\iff (iii)$.

To see that $(iii)\iff (iv)$, note that we can rewrite
\begin{equation}\label{e:siginflevsets}
4|\tilde{\sigma}_\infty|=\lim_{R\rightarrow \infty}\left|\frac{(\partial_zf)^2\partial_z\partial_zf-2\partial_zf\partial_{\bar{z}}f\partial_z\partial_{\bar{z}}f+(\partial_{\bar{z}}f)^2\partial_{\bar{z}}\partial_{\bar{z}}f}{|\partial_zf|^2}\right|_{z=Re^{iA}}
\end{equation}
The quantity on the right-hand side is the Hessian evaluated on the tangent to the level set, so its vanishing  at infinity means that the tangent becomes an asymptotic direction of the second fundamental form at infinity.

This completes the proof of Proposition \ref{p:4}.
\end{proof}

\vspace{0.1in}

\section{Homogeneous Polynomial Graphs}\label{s5}

In the case where the graph function is a homogeneous polynomial, Proposition \ref{p:4} says much more about the umbilic points at infinity.  Suppose that $P$ is given by $x^3=f(z,\bar{z})$, where $f$ is a real homogeneous polynomial of degree $n$.

Define 
\begin{equation}\label{d:hf}
{\mathcal H}(f)=(\partial_zf)^2\partial_z\partial_zf-2\partial_zf\partial_{\bar{z}}f\partial_z\partial_{\bar{z}}f+(\partial_{\bar{z}}f)^2\partial_{\bar{z}}\partial_{\bar{z}}f.
\end{equation}

This is a real homogeneous polynomial of degree $3n-4$ and and so, writing $z=x^1+ix^2$,
\[
{\mathcal H}(f)=(x^1)^{3n-4}\tilde{\mathcal H}(f)(\mu),
\]
where $\mu=x^2/x^1=\tan A$. Thus by Proposition \ref{p:4} and equation (\ref{e:siginflevsets}) we have:

\vspace{0.1in}
\begin{Prop}\label{p:umainf2}
A polynomial graph has an umbilic at infinity at $A$ iff 
\[
\tilde{\mathcal H}(f)(\tan A)=0.
\]
\end{Prop}
\vspace{0.1in}

\noindent{\bf Proof of Theorem \ref{t:1}}:

This follows from Proposition \ref{p:umainf2}, since the solutions of $\tilde{\mathcal H}(f)=0$ are isolated and there are at most $3n-4$ distinct solutions, each yielding two umbilic points at infinity: one at $A$ and one $A+\pi$, the antipodal direction.  
\qed

\vspace{0.1in}

\noindent{\bf Proof of Theorem \ref{t:2}}:

To prove this Theorem consider the following:

\vspace{0.1in}
\begin{Lem}
For any homogeneous real-valued polynomial $f$, the polynomial ${\mathcal H}(f)$ factorizes thus:
\[
{\mathcal H}(f)=-\frac{n}{4(n-1)}|{\mbox{Hess}}f|f,
\]
where $|{\mbox{Hess}}f|$ is the determinant of the Hessian of $f$.
\end{Lem}
\begin{proof}
Let $f$ be a real homogeneous polynomial of $(z,\bar{z})$ of degree $n$ and denote $\partial_z$ and $\partial_{\bar{z}}$ by $\partial$ and $\bar{\partial}$, respectively.

Recall  Euler's Lemma in complex coordinates:
\[
nf=z\partial f+\bar{z}\bar{\partial}f,
\]
its first derivative
\begin{equation}\label{e:euler1}
(n-1)\partial f=z\partial\partial f+\bar{z}\partial\bar{\partial}f,
\end{equation}
and their concatenation
\begin{equation}\label{e:euler2}
n(n-1) f=z^2\partial\partial f+2z\bar{z}\partial\bar{\partial}f+\bar{z}^2\bar{\partial}\bar{\partial}f.
\end{equation}
Consider Definition \ref{d:hf} written in the following way
\begin{align*}
(n-1)^2 {\mathcal H}(f) &= ((n-1)\bar{\partial} f)^2\partial \partial f-2((n-1)\partial f)((n-1)\bar{\partial}f)\partial_z\bar{\partial}f+((n-1)\partial_{\bar{z}}f)^2\bar{\partial}\bar{\partial}f\\
& = [\partial\partial f\bar{\partial}\bar{\partial}f-(\partial\bar{\partial}f)^2](z^2\partial\partial f+2z\bar{z}\partial\bar{\partial}f+\bar{z}^2\bar{\partial}\bar{\partial}f)\\
& = n(n-1)[\partial\partial f\bar{\partial}\bar{\partial}f-(\partial\bar{\partial}f)^2]f\\
& = -{\textstyle{\frac{1}{4}}}n(n-1)| \mbox{Hess} f |f,
\end{align*}
where we have used equation (\ref{e:euler1}) going from the first to the second line, equation (\ref{e:euler2}) going from the second to the third, and the determinant of the Hessian of $f$ is 
\[
| \mbox{Hess} f |
=\partial_{x}\partial_{x}f\partial_{y}\partial_{y}f-(\partial_{x}\partial_{y}f)^2=-4(\partial\partial f\bar{\partial}\bar{\partial}f-(\partial\bar{\partial}f)^2).
\]
Finally, rearranging gives the Lemma.

\end{proof}
\vspace{0.1in}

Thus for a homogeneous graph polynomial there is an umbilic point at infinity where either $f=0$ or $|{\mbox{Hess}}f|=0$. But these are precisely the points where the projective extension of the principal foliation of $P$ has a singularity \cite{gao}.

Let us now look at this more closely. The essential idea, which originated with Poincar\'e \cite{poinc}, is to double the principal foliations over the 2-sphere. To do this, view the domain ${\mathbb R}^2$ as the horizontal plane in ${\mathbb R}^3$ through the point $(0,0,1)$. A point $p\in{\mathbb R}^2$ is mapped to either of the points on the upper or lower hemispheres of the unit 2-sphere given by intersection with the associated line through the origin. In either case the points at infinity lie on the equator. 

The two principal foliations of $P$ are pulled back to the open hemispheres, these foliations having singularities at umbilic points, which appear in antipodal pairs on the 2-sphere.

For polynomial graphs, there is a unique projective extension of these foliations to the points at infinity, lying on the equator. The foliations are extended by considering the third fundamental form and extending this quadratic form.

For a homogeneous polynomial the foliations has singularities at infinity iff either $f=0$ or $|{\mbox{Hess}}(f)|=0$ (see Theorem 1.7 of \cite{gao}), that is, it has an umbilic at infinity as defined in this paper.

Since each umbilic point $p$ is isolated, the foliation has a well-defined index $i(p)$ which satisfies
\[
\sum_{p\in S^2}i(p)=\sum_{p\in S^2/S^1}i(p)+\sum_{\mbox{p umbilic }\in S^1}i(p)=\chi(S^2)=2,
\]
where here we count both the finite umbilic points and the umbilic points at infinity. Moreover, in \cite{gao} it is proven that the index of an umbilic at infinity is $1/2$

Thus the sum of the umbilic indices on $P$ (excluding antipodal points from the interior) and the indices at infinity, there must be at least one at umbilic point, since: 
\[
\sum_{p\in P}i(p)+\sum_{p\in S^1}i(p)\geq 1,
\]
which proves Theorem \ref{t:2}. 

\qed

\vspace{0.1in}

\section{Paraboloids}\label{s6}

To explore umbilic points at infinity explicitly, consider the paraboloids $P$ satisfying:
\[
x^3=\frac{(x^1)^2}{a}+\frac{(x^2)^2}{b},
\]
for non-zero constants $a,b$. If $ab>0$ it is the elliptic paraboloid, while if $ab<0$ it is the hyperbolic paraboloid.

Parametrize $P$ by the inverse of the Gauss map $\pi:S\rightarrow {\mathbb S}^2$, which in complex coordinates centred at the North pole $\xi=\tan({\textstyle{\frac{\theta}{2}}})e^{i\phi}$ is
\[
x^1=-{{\frac{a}{2}}}\frac{\xi+\bar{\xi}}{1-\xi\bar{\xi}}=-{\textstyle{\frac{a}{2}}}\cos\phi\tan\theta
\qquad\qquad
x^2=-{{\frac{b}{2i}}}\frac{\xi-\bar{\xi}}{1-\xi\bar{\xi}}=-{\textstyle{\frac{b}{2}}}\sin\phi\tan\theta
\]
\[
x^3={{\frac{a}{4}}}\frac{(\xi+\bar{\xi})^2}{(1-\xi\bar{\xi})^2}
+{{\frac{b}{4}}}\frac{(\xi-\bar{\xi})^2}{(1-\xi\bar{\xi})^2}
=({\textstyle{\frac{a}{4}}}\cos^2\phi+{\textstyle{\frac{b}{4}}}\sin^2\phi)\tan^2\theta.
\]
The coordinates span the interior of the unit disc: $0\leq\phi<2\pi$ and $0\leq\theta<\pi/2$ or $|\xi|<1$.

The support function is
\[
r=\frac{-a(\xi+\bar{\xi})^2+b(\xi-\bar{\xi})^2}{4(1+\xi\bar{\xi})(1-\xi\bar{\xi})}.
\]
The 2-parameter family of oriented normal lines to $S$ form a section of $T{\mathbb S}^2\rightarrow {\mathbb S}^2$, with $\xi\mapsto(\xi,\eta(\xi,\bar{\xi}))$ where
\[
\eta={\textstyle{\frac{1}{2}}}(1+\xi\bar{\xi})^2\partial_{\bar{\xi}} r=-\frac{a(\xi+\bar{\xi})(1+\xi^3\bar{\xi})+b(\xi-\bar{\xi})(1-\xi^3\bar{\xi})}{4(1-\xi\bar{\xi})^2}.
\]
Finally we compute
\[
\sigma=-\frac{(1-\xi\bar{\xi})[a(1+\bar{\xi}^2)^2-b(1-\bar{\xi}^2)^2]}{ab(1+\xi\bar{\xi})^3}.
\]

\vspace{0.1in}
\begin{Prop}
The paraboloid with $ab>0$ has two umbilic points of index a half if $a\neq b$, and a single umbilic point of index one if $a=b$. It has no umbilic points at infinity.

The paraboloid with $ab<0$ has no umbilic points on the open unit disc, but has four umbilic points at infinity, each of index one half.
\end{Prop}
\begin{proof}
The finite umbilic points occur where $\sigma=0$, which is
\[
\bar{\xi}^2=\frac{-a-b\pm2\sqrt{ab}}{a-b}.
\]
For $ab>0$ this yields four solutions, two of which are in the unit disk, which coincide when $a=b$. 

For $ab<0$ the solutions are at infinity $|\xi|^2=1$. We find that 
\[
{\mathcal H}(f)=-\frac{2}{ab}\left(\frac{x^2}{a}+\frac{y^2}{b}\right),
\]
and so if $ab>0$ there are no solutions and hence no umbilic points at infinity, while if $ab<0$ the solution set is a pair of lines, giving the four umbilic points at infinity.

The Gauss map at infinity is 
\[
\xi_\infty^2=e^{2i\theta_\infty}=\frac{(a-b)e^{-iA}-(a+b)e^{iA}}{(a-b)e^{iA}-(a+b)e^{-iA}},
\]
which has winding number $+1$ for $ab>0$ and $-1$ for $ab<0$.

The quantity $\tilde{\sigma}_\infty$ is
\[
\tilde{\sigma}_\infty=-\frac{1}{8ab}\left[a(1+e^{-2i\theta_\infty})^2-b(1-e^{-2i\theta_\infty})^2\right].
\]
To compute the index, consider the case $a=-b=1$ so that 
\[
\sigma=\frac{1-\xi\bar{\xi}}{2(1+\xi\bar{\xi})^3}(1+\bar{\xi}^4).
\]
The isolated umbilic points at infinity lie at fourth roots of negative unity: ${\xi}=\pm e^{\pm\pi i/4}$.

To compute the index at, say, ${\xi}=e^{\pi i/4}$, first use a linear change of variables to move the umbilic to the origin: $\xi'=\xi-e^{\pi i/4}$. Expanding $\sigma$ in these coordinates yields (dropping the prime)
\[
\sigma=G(\xi,\bar{\xi})e^{3\pi i/4}\bar{\xi}+O(2),
\]
where $G$ is a non-zero real function. Thus the index is one half, in agreement with part (iii) of Theorem 1.4 of \cite{gao}.
\end{proof}

\newpage

\noindent{\bf Statements and Declarations}:

The author has no relevant financial or non-financial interests to disclose. No data was collected in the course of this research.

\end{document}